\documentclass{ifacconf}

\usepackage{graphicx} 
\usepackage[numbers]{natbib}
\usepackage{mathtools}
\usepackage{amssymb,amsmath,amsfonts}
\usepackage{color}
\usepackage{microtype}
\usepackage{enumitem}
\usepackage{placeins}
\usepackage{comment}

\renewcommand{\theenumi}{\arabic{enumi}}

\DeclareMathOperator*{\argmin}{argmin}

\begin{document}

\begin{frontmatter}

\title{An Optimal Control Approach to Finite-Time Stabilization of Linear Systems}
%\thanksref{footnoteinfo}} 
% Title, preferably not more than 10 words.

\author{W. Weng }, 
\author{Y. Chitour }, 
\author{P. Mason }

\address{L2S, CNRS, Gif-sur-Yvette, France, (e-mail: weihao.weng@centralesupelec.fr, yacine.chitour@centralesupelec.fr, paolo.mason@centralesupelec.fr)}

\begin{abstract} 
    This paper presents a finite-time stabilization method for the integrator chain via optimal control techniques by introducing a tailored cost function. We analyze key properties of the resulting value function to derive the corresponding Hamilton-Jacobi-Bellman equation. Numerical simulations 
    illustrate the effectiveness of the approach.
\end{abstract}

\begin{keyword}
Finite-time stabilization, optimal control, linear system, Hamilton-Jacobi-Bellman equation %, gain regulation  
\end{keyword}

\end{frontmatter}
%===============================================================================

\newenvironment{sketch}{%
  \par\noindent\textbf{Sketch of proof.} \ignorespaces
}{\hfill$\square$\par}

\newenvironment{proof}{%
  \par\noindent\textbf{Proof.} \ignorespaces
}{\hfill$\square$\par}

\section{Introduction}
Finite-time stabilization has attracted considerable attention in nonlinear control due to its ability to achieve convergence in bounded time, which is desirable for systems that require fast and robust responses. Among classical benchmarks, the integrator chain serves as a fundamental model for studying nonlinear feedback design and optimal control strategies. Moreover, finite-time stabilization of a chain of integrators arises in several control applications. Perturbed double integrators, for instance, serve as models for motorized actuators or robotic arms (see e.g. \cite{MR1618028,hong2002finiterobot,orlov2009discontinuous}), while Higher Order Sliding Mode Control (cf. \cite{emel1996high, shtessel2014sliding, cruz2017homogeneous}) can be formulated as the stabilization of an auxiliary system arising as a perturbed chain of integrators built from the output and its higher time derivatives as shown by \cite{MR2889285}. 
 The by now classical approaches to handle finite-time stabilization of the chain of integrators include homogeneous feedback design, Lyapunov-based synthesis, and sliding mode techniques. This issue
  was addressed in relation with homogeneous systems in \cite{zubov1961methods, bhat1997finite}, and homogeneity was used for stabilization of linear systems in \cite{praly1997generalized}. In \cite{MR2150956}, the link between finite-time stabilization and homogeneity of a system was established, and it was shown that a homogeneous system is finite-time stable if and only if it is asymptotically stable and has a negative homogeneity degree. This result has, since then, been used for the development of many controllers for a chain of integrators. A homogeneous nonsmooth proportional-derivative controller for robot manipulators (double integrator system) was developed in \cite{hong2002finiterobot}, and this approach was generalized for an arbitrary length of a chain of integrators in \cite{hong2002finite}. 
The homogeneity approach was also used in \cite{MR1999375,levant2005homogeneity} to obtain finite-time stabilization of the arbitrary order sliding mode controllers for Single Input Single Output (SISO) systems \cite{levant2001higher}. A recent approach introduced by \cite{polyakov2015finite} exploits the concept of implicit Lyapunov function in order to ensure finite-time stabilization of the perturbed integrator chain.
Another approach relies on optimal control and has been initiated by Nakamura et al., 
\cite{nakamura2009homogeneous,nakamura2010global,NAKAMURA2013552}), see also \cite{mboup2017remark} and \cite{haddad2015finite}.
In this paper, we aim at addressing the issue of finite-time stabilization of a chain of integrators relying in a systematic way on the use of optimal control. 
Inspired by previous works, where homogeneity properties of the Lyapunov functions played a crucial role for finite-time stabilization,  we consider a large class of homogeneous cost functions to be minimized over an infinite horizon. We carefully study the associated value function, establishing in particular a homogeneity property of the latter. As a consequence, we prove firstly that convergence to zero occurs in finite time and, secondly, that the value function is regular enough so that it satisfies a corresponding Hamilton-Jacobi-Bellman (HJB) equation. As a consequence, the optimal control can be expressed in feedback form using the value function. Furthermore, we show that some controllers achieving finite-time stabilization from the literature (cf. \cite{hong2002finite} and  \cite{harmouche2017stabilisation}) are particular instances of our framework. Numerical simulations %and an application to a gain regulation problem 
illustrate the approach.

The paper is organized as follows. In Section~\ref{sec-ocp} we formulate the optimal control problem under study and discuss some related basic results. In Section~\ref{sec-main} we show that the optimal trajectories converge to the origin in finite time, we derive the HJB equation and discuss regularity properties of the value function. In Section~\ref{sec-simu} we provide numerical simulations, while Section~\ref{sec-conclusion} concludes the paper.

\subsection{Notations}
%\begin{itemize}
%    \item $J_n$ is the Jordan block of size $n$ with $0$ as eigenvalue, $(e_i)_{1\leq i\leq n}$ designate %the canonical basis of $\mathbb{R}^n$.
%    \item For $1 \leq j\leq n $, $\bar{x}_j=(x_1,\dots,x_j)$.
%    \item For $I \subset \mathbb{R}, k > 0$, $\|\cdot\|_{L^{k}(I,\mathbb{R}^n)}: u \mapsto %\left(\int_I\|u(t)\|^{k}dt\right)^\frac{1}{k}.$
%    \item For $r > 0$, $\lfloor\cdot\rceil^r : x \mapsto |x|^rsgn(x)$, where $sgn$ is the sign function.
%\end{itemize}
 We denote by $J_n$ the Jordan block of size $n$ corresponding to the zero eigenvalue, and by $(e_i)_{1\leq i\leq n}$ the canonical basis of $\mathbb{R}^n$.
 
 For $r > 0$,  we define $\lfloor\cdot\rceil^r : x \mapsto |x|^rsgn(x)$, where $sgn$ is the sign function.
 
We introduce below some definitions that will be used extensively in the upcoming sections.
\begin{defn}
    \label{def1}
    Let $m \in \mathbb{N}_{\ge 1},r=(r_1,...,r_m) \in \mathbb{R}_{>0}^m$,  and $\epsilon > 0$. We call $r$-weighted dilation the mapping from $\mathbb{R}^m$ to itself defined as
    \[\delta_\epsilon^r(x_1,x_2,...,x_m) = (\epsilon^{r_1}x_1, \epsilon^{r_2}x_2,...,\epsilon^{r_m}x_m).\]
\end{defn}

\begin{defn}
    \label{def2}
    Let $m \in \mathbb{N}_{\ge 1},r\in \mathbb{R}_{>0}^m$,  and  $\mu \in \mathbb{R}$. A mapping $f : \mathbb{R}^m \to \mathbb{R}$ is said to be homogeneous of degree $\mu$ with respect to the family of dilations $(\delta^r_\epsilon)_{\epsilon > 0}$ if 
     \[\forall \epsilon > 0,\ \forall x \in \mathbb{R}^m,\quad f(\delta_\epsilon^r(x))= \epsilon^{\mu} f(x).\]
\end{defn}

\begin{defn}
    \label{def3}
    Let $m \in \mathbb{N}_{\ge 1},r\in \mathbb{R}_{>0}^m$,  and  $\mu \in \mathbb{R}$.  A vector field $\mathcal{G}:\mathbb{R}^m \rightarrow \mathbb{R}^m$ is said to be weighted homogeneous of degree $\mu$ with respect to the family of dilations $(\delta^r_\epsilon)_{\epsilon > 0}$ if 
    \[\forall \epsilon > 0,\ \forall x \in \mathbb{R}^m,\quad \mathcal{G}(\delta_\epsilon^r(x))= \epsilon^\mu\delta^r_\epsilon(\mathcal{G}(x)).\]
\end{defn}

\section{Problem Formulation and basic results}
\label{sec-ocp}
In this section, we formulate the problem of interest and present some preliminary results.

\subsection{Statement of the problem}

 Let $n, m \in \mathbb{N}_{\ge 1}$ and $q > 1$, consider a pair of matrices $(A,B) \in \mathbb{R}^{n\times n}\times \mathbb{R}^{n\times m}$ as well as $x_0 \in \mathbb{R}^n$ and the autonomous system with $u \in L^{q}(\mathbb{R}_{\ge 0},\mathbb{R}^m)$,
\begin{equation}
\label{eq:L}
\qquad  \forall t \geq 0,\quad 
\begin{cases}
    \dot{x}(t)=Ax(t)+Bu(t),\\
    x(0)=x_0.
\end{cases}
\end{equation}

Introduce the instantaneous cost
\begin{equation}\label{eq:inst-cost}
\mathcal{L} : (x,u) \in \mathbb{R}^n\times\mathbb{R}^m \mapsto \frac{\|u\|^{q}}{q} +F(x),
\end{equation}
where $F$ is a mapping from $\mathbb{R}^n$ to $\mathbb{R}_{\ge 0}$ satisfying the following assumptions:
\begin{enumerate}[label= (A\arabic*), leftmargin=1cm]
    \item $F \in C^1(\mathbb{R}^n,\mathbb{R})$, i.e., is continuously differentiable.
    \item $F$ is convex and positive definite, i.e., $F(0) = 0$ and $F(x)>0$ for every $x\neq 0$.
    \item There exist $r\in \mathbb{R}_{>0}^n$, $d > \max_{1 \leq j\leq n} r_j$ such that
    \[\forall \epsilon > 0, \forall x \in \mathbb{R}^n\text{, }\quad F(\delta_\epsilon^r(x)) = \epsilon^dF(x).\]
\end{enumerate}
%Define 
%\begin{equation}\label{eq:inst-cost}
%\mathcal{L} : (x,u) \in \mathbb{R}^n\times\mathbb{R}^m \mapsto \frac{\|u\|^{q}}{q} +F(x).
%\end{equation}
For $T\in (0,\infty]$, the optimal control problem \((OCP)_T\) consists in minimizing the integral cost %$\mathcal{J}_T:\mathbb{R}^n\times L^{q}([0,T),\mathbb{R}^m)\to [0,\infty]$ defined as
\[
\mathcal{J}_T:\mathbb{R}^n\times L^{q}([0,T),\mathbb{R}^m)\longrightarrow [0,\infty] %\nonumber\\
\]
defined as 
\begin{equation}
\mathcal{J}_T(x_0,u) := \int_0^T \mathcal{L}(x(t),u(t))dt,\label{eq:JT}
\end{equation}
for trajectories of~\eqref{eq:L} starting at any $x_0\in\mathbb{R}^n$. We set 
\[V_T(x_0) = \operatorname*{inf}_{u \in L^{q}([0,T),\mathbb{R}^m)} \mathcal{J}_T(x_0,u).
\]
The function $V_T$ is the \emph{value function} associated with \((OCP)_T\). Note that  $\mathcal{J}_{T_1}(x_0,u) < \mathcal{J}_{T_2}(x_0,u)$ for every $T_1<T_2$ in $(0,\infty]$ and, as a consequence  $V_{T_1}\leq V_{T_2}$.

%For \(0<T<\infty\), the finite horizon problem %\((OCP)_T\) is defined analogously by
%\[
%\mathcal{J}_T(x_0,u) \;=\; \int_0^T %\mathcal{L}(x(t),u(t))dt,
%\quad u \in L^{q}([0,T],\mathbb{R}^m).
%\]

%\subsection{Existence and uniqueness of a minimizer}

\subsection{Preliminary results concerning $(OCP)_T$}

%{\color{blue} completare il teorema con la convessità del costo}
%Convexity of the cost and value function, as well as 
From the convexity of the instantaneous cost $\mathcal{L}$ it easily follows that the integral cost $\mathcal{J}_T$ is also convex, as stated below.
\begin{prop}
\label{p:conv}
For $T\in \mathbb{R}_{>0}$ the function  $\mathcal{J}_T$ is convex  on $\mathbb{R}^n\times L^{q}([0,T),\mathbb{R}^m)$. If $T = \infty$ convexity holds on the domain of $\mathcal{J}_T$, i.e., the subset of $\mathbb{R}^n\times L^{q}([0,T),\mathbb{R}^m)$ such that $\mathcal{J}_T$ is bounded.
\end{prop}
Uniqueness of the minimizer for $(OCP)_T$ for every $x_0\in \mathbb{R}^n$ can be easily obtained from Proposition~\ref{p:conv}. On the other hand, existence of the minimizer follows from classical arguments (see e.g.~\cite[Theorem 11]{lee1967foundations}).
\begin{thm}
\label{thm1}
Let $T\in \mathbb{R}_{>0}$  or $T=\infty$ and $(A,B)$ controllable. Then for every $x_0\in\mathbb{R}^n$,  $(OCP)_T$ admits a unique minimizer in $L^{q}([0,T),\mathbb{R}^m)$. 
%For $0<T <\infty$, there exists an unique minimizer for $(OCP)_T$ i.e., $V_T(x_0)$
%is achieved at a unique $u\in L^{q}([0,T),\mathbb{R}^m)$. The same result is true for $T=\infty$ if, moreover, $(A,B)$ is controllable. 
\end{thm}
%\begin{sketch} 
%Fix $T < \infty$,  the existence of a solution for $(OCP)_T$ is done using the direct method in the calculus of variations. Uniqueness of the resulting control follows from the strict convexity of the cost in the control variable.
%The proof for the existence and uniqueness of a minimizer for $(OCP)_\infty$ follows from the finite horizon problem and is done in \cite[Theorem 11]{lee1967foundations}.
%\end{sketch}
As a consequence of the previous result, we have the following.
\begin{prop}
\label{p:conv2}
If $T\in \mathbb{R}_{>0}$ or $T=\infty$ and $(A,B)$ is controllable then the function $V_T$ takes values in $\mathbb{R}$ and is convex. In particular, $V_T$ is continuous.
\end{prop}
%\subsection{Preliminary results}
The following classical result (see e.g. \cite[Theorem 3.2]{hestenes1966calculus} for a more general formulation) %expresses an equivalence property among continuous positive-definite functions which are homogeneous with respect to the same family of dilations. 
expresses the fact that continuous positive-definite functions which are homogeneous with respect to the same family of dilations are commensurable.
\begin{lem}
    \label{lem1}
     Let $k \in \mathbb{N}_{\ge 1}$ and $W_1, W_2$ be two positive definite and continuous mappings from $\mathbb{R}^k$ to $\mathbb{R}_{\ge 0}$. Suppose 
    there exist $s \in \mathbb{R}_{>0}^k$ and $d_1,d_2 >0$ such that $W_1$ and $W_2$ are homogeneous of degree $d_1$ and $d_2$, respectively, with respect to the family of dilations $(\delta_{\epsilon}^s)_{\epsilon>0}$.
%    \[\forall i\in \lbrace1,2\rbrace,\forall \epsilon > 0, \forall x \in \mathbb{R}^m, \quad W_i(\delta^s_\epsilon(x))=\epsilon^{d_i}W_i(x).\]
     Then, there exist $c_1,c_2 > 0$ such that
     \[\forall x \in \mathbb{R}^k , \quad c_1W_1^\frac{d_2}{d_1}(x) \leq W_2(x) \leq c_2W_1^\frac{d_2}{d_1}(x).\]
     The inequality on the right remains true even if the assumption that $W_2$ is positive definite is dropped.
\end{lem}
%\begin{sketch}
%    A general version of this lemma is proved in \cite[Theorem 3.2]{hestenes1966calculus}.
%\end{sketch}
%From the above lemma, we deduce a homogeneous upper-bound for the mapping $F$.
Since $F$ is continuously differentiable, by applying the previous lemma and the chain rule, we obtain the following result.
\begin{cor}
    \label{cor1}
    Let $\alpha_i=\frac{d}{r_i}$ and $\alpha_{i,j}=\frac{d-r_i}{r_j}$ for $i,j =1,\dots,n$. Then
    \begin{enumerate}
    \setlength{\itemsep}{5pt}
        \item There exist two constants $c_1,c_2 > 0$ %and a family $(\alpha_i)_{1 \leq i \leq n} \in \mathbb{R}_{>0}^n$ 
        such that 
    \[\forall x \in \mathbb{R}^n\text{, }  \quad  c_1\sum_{i=1}^n |x_i|^{\alpha_i}\leq F(x) \leq c_2\sum_{i=1}^n |x_i|^{\alpha_i}.\]
        \item The function $\partial_i F$ is homogeneous of degree $d-r_i$ with respect to $(\delta_{\epsilon}^r)_{\epsilon>0}$ and, for any $1\leq i \leq n$, there exists a constant $c_{i}> 0$ %and a family $(\alpha_{i,j})_{1\leq j\leq n} \in \mathbb{R}_{>0}^n$ 
        such that 
       \[\forall x \in \mathbb{R}^n\text{, }  \quad|\partial_i F(x)| \leq  c_i\sum_{j=1}^n |x_j|^{\alpha_{i,j}}.\]
        \end{enumerate}
\end{cor}
%\begin{sketch}
%For the first item, it is a direct application of Lemma~\ref{lem1} with $\alpha_i=\frac{d}{r_i}-1$, $1\leq i \leq n$, while for the second we take  $\alpha_{i,j}=\frac{d-r_i}{r_j}$, $1\leq i,j \leq n$.
%   
%\end{sketch}
If $(A, B)$ is controllable, then it can be stabilized exponentially by means of a linear feedback. As a consequence, one can easily upper-bound 
%One can upper-bound $V_T$ and 
$V_\infty$ with a term that depends exclusively on the initial data.
\begin{prop}
    \label{prop1}
   If $(A,B)$ is controllable then there exists $C > 0$ such that
   % \[\forall x \in \mathbb{R}^n, \quad V_\infty(x) \leq C\|x\|\max\lbrace\|x\|^{q-1}, \|x\|^{\alpha_-},\|x_0\|^{\alpha^+}\rbrace,\]
   % with $\alpha_- = \min_{1\leq j\leq n} \frac{d}{r_j}-1$ and $\alpha^+=\max_{1\leq j\leq n} \frac{d}{r_j}-1$.
   \[V_T(x)\leq V_\infty(x) \leq C \max\{\|x\|^{\alpha_-},\|x\|^{\alpha_+}\}\]
   with $\alpha_- = \min\{q,\frac{d}{r_1},\dots,\frac{d}{r_n}\}$ and $\alpha_+= \max\{q,\frac{d}{r_1},\dots,\frac{d}{r_n}\}$.
\end{prop}
%\begin{sketch}
%    Consider $K \in \mathbb{R}^{n}$ such that $M\coloneqq A+BK^\top$ is Hurwitz, and introduce the feedback control $u_K=K^\top x$. Consider the associated trajectory $x_K(t):= e^{Mt}x_0$ and note from Corollary~\ref{cor1} that  $\mathcal{L}(x_K(t),u_K(t)) = \mathcal{L}(e^{Mt} x_0, K^\top e^{Mt} x_0) \leq C_0 e^{\lambda t} \max\{\|x_0\|^{\alpha_-},\|x_0\|^{\alpha_+}\}$, where $\lambda$ is any negative value larger than the spectral abscissa of $M$ and $C_0$ is a suitable positive constant. The result then follows from the estimate $V_T(x_0)\leq V_\infty(x_0) \leq  
%%    \mathcal{J}_\infty(x_0,u_K)$.
    %Substitute for $u_K$ and the associated trajectory $x_K: t \mapsto e^{Mt}x_0$ such that $V_\infty(x_0) \leq \int_0^\infty \mathcal{L}(u_K(t),x_K(t))dt$. Finally, use H\"older's inequality as well as Corollary~\ref{cor1} and standard upper-bound on $x_K$ to get the desired upper-bound.
%\end{sketch}
\subsection{Hamiltonian approach}
From here on, let us take $A=J_n$ and $B=e_n$, i.e., the system~\eqref{eq:L} encodes an integrator chain and $m=1$. For $T < \infty$ the solutions of $(OCP)_T$ satisfy the Maximum Principle (see e.g.~\cite{lee1967foundations}). More precisely, introduce the Hamiltonian $H$ of $(OCP)_T$ defined on $\mathbb{R}^n\times \mathbb{R}\times\mathbb{R}^n$ %\times\mathbb{R}_{\ge 0} $
by 
\[H(x,u,p)=p^\top(J_nx+ue_n)-\left(
\frac{|u|^{q}}{q}+F(x)\right).\]
Denoting by $u_T$ the minimizing control for $(OCP)_T$ with initial condition $x_0$, and by $x_T$ the associated trajectory, %associated trajectory $x_T$, 
the Maximum Principle establishes the existence of an absolutely continuous mapping $p_T$ such that, for any $t \in [0, T]$,
%\begin{align}
%    u_T(t)&=\argmin_{u\in\mathbb{R}} H(x_T(t),u,p_T(t)) = \lfloor p_T(t)_n\rceil^\frac{1}{q-1}, \label{eq1}\\[0.5em]
%    \dot{x}_T(t)&=\frac{\partial H}{\partial p} (x_T(t),u,p_T(t)) = J_nx_T(t)+u_T(t) e_n, \label{eq2}\\[0.5em]
%    \dot{p}_T(t)&=-\frac{\partial H}{\partial x} (x_T(t),u,p_T(t)) =-J_n^\top p_T(t)+\nabla F(x_T(t)), \label{eq3}\\[0.5em]
%    p_T(T)&=0. \label{eq4}
%\end{align}
\begin{equation}
u_T(t)=\argmin_{u\in\mathbb{R}} H(x_T(t),u,p_T(t)) = \lfloor p_T(t)_n\rceil^\frac{1}{q-1}, \label{eq1}
\end{equation}
and $(x_T,p_T)$ solves the Hamiltonian equations %\[(\dot x,\dot p) = \left(\frac{\partial H}{\partial p}(x,u_T,p),\frac{\partial H}{\partial p}(x,u_T,p)\right), i.e.,\]
\begin{align}
   \dot{x}& = J_n x+ \lfloor p_n\rceil^\frac{1}{q-1} e_n, \label{eq2}\\
    \dot{p}& =-J_n^\top p+\nabla F(x), \label{eq3}
\end{align}
with terminal (transversality) condition
\begin{equation}
    p(T)=0. \label{eq4}
\end{equation}
In the following, we will refer to the pair $(x_T,p_T)$ as the \emph{extremal lift} associated with the optimal trajectory $x_T$.
More generally, solutions of~\eqref{eq2}-\eqref{eq3} are said to be \emph{extremal pairs}.
%, and to any pair $(x_T,p_T)$ as \emph{extremal lift} associated with the optimal trajectory $x_T$.
%\begin{rem}

Since $F$ is only assumed to be of class $C^1$, the  equations~\eqref{eq2}-\eqref{eq3} do not necessarily satisfy the usual local Lipschitzness assumption guaranteeing uniqueness of the solution from any given initial condition. 
%In particular, it can be shown that a solution $(x,p)$ of~\eqref{eq2}-\eqref{eq3} is the unique solution (in positive time) starting from $(x(0),p(0))$ if and only if $(x(t),p(t))\neq (0,0)$ for every $t\geq 0$.
%\end{rem}
However, we establish a uniqueness result for solutions of~\eqref{eq2}--\eqref{eq4}.
\begin{prop}
    \label{prop5}
    Let $T > 0$ and $x_0 \in \mathbb{R}^n$. Then, the equations~\eqref{eq2}--\eqref{eq4} admit a unique solution. %, the problem $(OCP)_T$ admits a unique [* extremal *] solution.
\end{prop}
\begin{proof}
    Let $T >0$ and $x_0 \in \mathbb{R}^n$. Suppose that $(\bar{x},\bar{p})$ and $(x,p)$ are solutions of~\eqref{eq2}--\eqref{eq4}. Setting \[\Phi(t):=(x(t)-\bar x(t))^\top (p(t)-\bar p(t)),\] a direct computation shows that
   \begin{align*}
 \dot  \Phi& =  (\lfloor p_n \rceil^{\frac{1}{q-1}}-\lfloor \bar{p}_n \rceil^{\frac{1}{q-1}})(p_n-\bar{p}_n)\\
   & \ + (x-\bar x)^\top (\nabla F(x) - \nabla F(\bar x)).
   \end{align*} 
    The first term on the right-hand side of the previous equality is nonnegative independently of $p_n,\bar{p}_n$, %, i.e., if the extremal trajectories correspond to the same control input. 
while the second term is always nonnegative as a consequence of the convexity of $F$, so that $\dot \Phi\geq 0$. Moreover $\Phi(0)=\Phi(T)=0$, from which we deduce that $\dot \Phi$ is identically equal to zero on $[0,T]$. We deduce that $p_n=\bar{p}_n$, i.e., the extremal trajectories correspond to the same control input. In turn, this implies that $x=\bar x$ and, from~\eqref{eq3}, that $\frac{d}{dt}(p-\bar p) = -J_n^\top (p-\bar p)$. By~\eqref{eq4} we deduce that $p=\bar p$, concluding the proof of the proposition.
%     %[* extremal *] solutions of $(OCP)_T$ and denote by $\bar{u}$ (resp. $u$) the optimal control associated with the pair $(\bar{x},\bar{p})$ (resp. $(x,p)$). Using equations (\ref{eq1})--(\ref{eq3}) as well as the convexity of the mapping $F$, one establishes the following upper-bound
%    \begin{align}
%        \int_0^T\!(\bar{u}(t)\!-\!u(t))(\lfloor\bar{u}(t)\rceil^{q-1}\!-\!\lfloor u(t)\rceil^{q-1})dt &\leq \langle  \Delta x(T), \Delta p(T) \rangle \label{eq21}
%    \end{align}
%    where $\Delta x(T)=\bar{x}(T)-x(T)$ and $\Delta p(T)= \bar{p}(T)-p(T)$.
%    Moreover, one can show that $x(T)=0 \implies p(T)=0$. Using (\ref{eq4}) i.e. the transversal condition, $\bar{p}(T)=0$ so deduce that the right-hand side of (\ref{eq21}) equals 0. Positivity of the left-hand side of (\ref{eq21}) implies that it also equals $0$. Since the mapping $(\bar{u}-u)(\lfloor\bar{u}\rceil^{q-1}-\lfloor u\rceil^{q-1})$ is continuous and positive whose integral on $[0,T]$ is zero, we deduce that $\bar{u} \equiv u$ on $[0,T]$.
%    Since both $\bar{x}$ and $x$ are solutions of~\eqref{eq:L}
%    then if $\xi \coloneqq \bar{x}-x$ is such that
%    for $t \in [0,T]$, $\dot{\xi}(t)=J_n\xi(t)$ and $\xi(0)=0$
%    so that $\xi \equiv 0$ i.e. $\bar{x} \equiv x$.
%    Similarly, considering the difference $\bar{p}-p$ and exploiting (\ref{eq3}), one shows that $\bar{p}\equiv p$ on $[0,T]$.
\end{proof}

Next, we adapt the conditions of the Maximum Principle to the infinite-time horizon problem. Its proof will be given in a forthcoming publication.
\begin{thm}
\label{thm2}
Let $u_{\infty}$ and $x_\infty$ be the optimal control and associated trajectory, respectively, for $(OCP)_\infty$ starting from $x_0$. Then $x_\infty$ admits a unique extremal lift $(x_\infty,p_\infty)$, i.e. a unique solution of the Hamiltonian dynamics \eqref{eq2}-\eqref{eq3}, and $u_{\infty}$ satisfies~\eqref{eq1}.
%$p^\infty_0 \in \mathbb{R}^n$ such that the pair $(x_\infty(\cdot),p_\infty(\cdot))$ is a solution of the Hamiltonian dynamics (\ref{eq1})--(\ref{eq3}) starting from $(x_0,p_0^\infty)$.
\end{thm}
Note that by Theorem~\ref{thm2} the optimal control is continuous, so that chattering is not present.

\section{Main Results}
\label{sec-main}
In this section, we shall investigate and exploit the properties of the value function to prove a result on finite-time stabilization of the integrator chain.

\subsection{Finite-time stabilization}
Introduce the Hamiltonian vector field
\[\mathcal{F}: (x,y) \in \mathbb{R}^n\times\mathbb{R}^n \mapsto \begin{pmatrix}
    J_n x+e_n\lfloor y_n\rceil^\frac{1}{q-1} \\
    -J_n^\top y +\nabla F(x)
\end{pmatrix}.\]
The following result can be shown by straightforward computations.
\begin{prop}
    \label{prop3}
    Let $d,\mu \in \mathbb{R}_{>0}$ such that $\mu < \frac{d(q-1)}{qn}$. In accordance with Assumption (A3), assume that $F$ is homogeneous of degree $d$ with respect to the family $(\delta^r_\epsilon)_{\epsilon>0}$ where  the weights $r_i$ are defined as
   % Introduce    \[r_1 = \frac{d}{q} +n\mu \qquad \text{ and } \qquad s_1 = \frac{d(q-1)}{q} -(n-1)\mu.\]
 %   Define the following finite sequences :
     \begin{equation}
        \forall i \in \lbrace1,\dots,n\rbrace\text{, } \quad r_i =\frac{d}{q}+(n-i+1)\mu.\label{eq12}
            \end{equation}
Furthermore, define
             \begin{equation}
        \forall i \in \lbrace1,\dots,n \rbrace\text{, } \quad s_i=  \frac{d(q-1)}{q}-(n-i)\mu.
        \label{eq13}
    \end{equation}
%    \begin{align}
%        \forall i \in \lbrace1,\dots,n\rbrace\text{, } \quad r_i =\frac{d}{q}+(n-i+1)\mu,\label{eq12}\\[1em]
%        \forall i \in \lbrace1,\dots,n \rbrace\text{, } \quad s_i=  \frac{d(q-1)}{q}-(n-i)\mu.
%        \label{eq13}
%    \end{align}
    Then, the vector field $\mathcal{F}$ is homogeneous of degree $-\mu$ with respect to the family 
    $(\delta^{(r,s)}_\epsilon)_{\epsilon > 0}$, where $\delta^{(r,s)}_\epsilon(x,p):= (\delta^r_\epsilon(x), \delta^s_\epsilon(p))$.
\end{prop}

%\begin{sketch}
  %  Let $r_1,s_1,\mu$ as in the statement of the proposition. The equations (\ref{eq12})--(\ref{eq13}) are obtained by investigating the equality from Definition~\ref{def3} for $\mathcal{F}$.
%\end{sketch}
%Let $\epsilon > 0$ and $\delta^{(r,s)}_\epsilon$ be the family of dilation such that $r,s$ satisfy (\ref{eq12})--(\ref{eq13}) given by Proposition~\ref{prop3}.
Set $\tilde{\mathcal{L}}(x,p) = \mathcal{L}(x,\lfloor p_n\rceil^{\frac{1}{q-1}})$, where $\mathcal{L}$ is  the instantaneous cost defined in \eqref{eq:inst-cost}. From the homogeneity of $F$ with respect to $(\delta^r_\epsilon)_{\epsilon>0}$ and the definition of $s_n$ in~\eqref{eq13}, one easily gets that $\tilde{\mathcal{L}}$ is homogeneous of degree $d$ with respect to the family of dilations $(\delta^{(r,s)}_\epsilon)_{\epsilon>0}$.
%Let $z=(x,p)$ be an extremal lift, and let us define
%a pair solution of the equations $(\ref{eq1})$--$(\ref{eq3})$, let us also define 
%\[z_\epsilon : t \mapsto \delta^{(r,s)}_\epsilon(z(\epsilon^{-\mu} t)).\]
We next show a homogeneity property of the value function $V_\infty$.
\begin{lem}
    \label{lem4}
         Let $z_0=(x_0,p_0) \in \mathbb{R}^n\times\mathbb{R}^n$ and $z=(x,p)$ %a solution of (\ref{eq1})--(\ref{eq3}) 
         be an extremal pair starting from $z_0$. Consider the map $z_\epsilon(t):=  \delta^{(r,s)}_\epsilon(z(\epsilon^{-\mu} t))$. The following properties hold:
        \begin{enumerate}
        \setlength{\itemsep}{5pt}
            \item $z_\epsilon$ is an extremal pair %solution of (\ref{eq1})--(\ref{eq3}) 
            starting from $\delta^{(r,s)}_\epsilon(z_0)$.
            \item  The mapping from $\mathbb{R}^n\times\mathbb{R}^n$ to $[0,\infty]$  given by
        \[\mathcal{V}(z_0)= \int_0^\infty \tilde{\mathcal{L}}(z(\tau))d\tau\]
        \noindent
        is homogeneous of degree $d+\mu$ with respect to $(\delta^{(r,s)}_\epsilon)_{\epsilon>0}$.
        Moreover, $V_\infty$ is homogeneous of degree $d +\mu$ with respect to  $(\delta^r_\epsilon)_{\epsilon>0}$, and $z$ is the extremal lift of an optimal trajectory if and only if $z_\epsilon$ is.
        \end{enumerate}
\end{lem}
\begin{proof}
Item 1. easily follows from Proposition~\ref{prop3} and the homogeneity of $\tilde{\mathcal{L}}$. Concerning Item 2., one has
\begin{align*}
\mathcal{V}(\delta^{(r,s)}_\epsilon(z_0))&= \int_0^\infty \tilde{\mathcal{L}}(z_\epsilon(\tau))d\tau = \epsilon^d \int_0^\infty \tilde{\mathcal{L}}(z(\epsilon^{-\mu}\tau))d\tau\\
& = \epsilon^{d+\mu} \int_0^\infty \tilde{\mathcal{L}}(z(\tau'))d\tau'= \epsilon^{d+\mu} \mathcal{V}(z_0).
\end{align*}
%where we have performed the time rescaling $\tau' = e^{-\mu}\tau$.
%This proves the homogeneity of $\mathcal{V}$.
The remaining claims follow by observing that $V_\infty(x_0) = \inf_{p_0\in\mathbb{R}^n} \mathcal{V}(x_0,p_0)$ for every $x_0\in\mathbb{R}^n$ and $\delta_\epsilon^s:\mathbb{R}^n\to\mathbb{R}^n$ is bijective, so that
\begin{align*}
V_\infty(\delta^r_\epsilon(x_0))& = \inf_{p_0\in\mathbb{R}^n} \mathcal{V}(\delta^r_\epsilon(x_0),p_0) = \inf_{\tilde{p}_0\in\mathbb{R}^n} \mathcal{V}(\delta^{(r,s)}_\epsilon(x_0,\tilde{p}_0))\\
& =\epsilon^{d+\mu} \inf_{\tilde{p}_0\in\mathbb{R}^n} \mathcal{V}(x_0,\tilde{p}_0) = \epsilon^{d+\mu} V_\infty(x_0).
\end{align*}
%    \begin{enumerate}
%        \item It is a direct computation using Proposition~\ref{prop3}.
%        \item The result follows using the homogeneity of $\mathcal{L}$ with respect to $\delta^r_\epsilon$ and a time rescaling (change of variable) $\tau= \epsilon^{-\mu}t$.
%    \end{enumerate}
\end{proof}
We have the necessary tools to prove the following theorem.
\begin{thm}
    \label{thm3}
    Let  $x_0\in \mathbb{R}^n$ and consider the solution $x_\infty$ of $(OCP)_\infty$ starting at $x_0$. Then $x_{\infty}$ reaches the origin in finite time. Moreover, if $T_{x_0}$ is defined as the first time such that $x_\infty(T_{x_0}) = 0$,
 %   The optimal trajectory associated with the solution of the Hamiltonian dynamic given by (\ref{eq1})--(\ref{eq3}) minimizing $\mathcal{J}_\infty$ reaches the origin in finite time. Denote by $T_{x_0}$ the first time when the optimal trajectory solution of $(OCP)_\infty$ reaches $0$, i.e., if $T_{x_0}$ is the infimum of times $T>0$ so that 
  %  $\forall t \geq T$, $x(t)=0$, 
     there exists $C > 0$ such that
    \begin{equation}\label{eq:TCV}
    T_{x_0} \leq C\sum_{i=1}^n |(x_0)_i|^{\frac{\mu}{r_i}}
    %\frac{V^{1-\frac{d}{d+\mu}}_q(x_0)}{c(1-\frac{d}{d+\mu})}.
    \end{equation}
\end{thm}
\begin{proof}
Consider an extension of the map $\mathcal{V}$ in Lemma~\ref{lem4} defined as 
\[\mathcal{W}(z_0,t) =\int_t^\infty \tilde{\mathcal{L}}(z(\tau))d\tau,\]
where $z$ is an extremal pair starting from $z_0$. From now on we assume that $z$ is the extremal lift associated with the optimal trajectory $x_\infty$. Since, for every $\tau\geq 0$, the curve  $x_\infty(\tau+\cdot)$ is the solution of $(OCP)_\infty$ starting from $x_\infty(\tau)$, we have that 
\begin{equation}
\label{eq-V-W}
V_\infty(x_\infty(s+t)) = \mathcal{W}(z_0,s+t) = \mathcal{W}(z(s),t)
\end{equation} 
for every $s,t\geq 0$. The function $\mathcal{W}$ is differentiable with respect to the second variable, and we obtain from~\eqref{eq-V-W}
\begin{align*}
\dot{V}_{\infty}(x_{\infty}(s))& = %\frac{d\mathcal{W}}{dt}(z_0,s) = 
\frac{d\mathcal{W}}{dt}(z(s),0)\\
& = - \tilde{\mathcal{L}}(z(s))\\
& \leq -F(x_\infty(s)).
\end{align*}
From the homogeneity of $V_\infty$ given by Lemma~\ref{lem4}, as well as the positive-definiteness and homogeneity of $F$, we obtain from Lemma~\ref{lem1} that
\[\dot{V}_{\infty}(x_{\infty}(s)) \leq -c\, V_\infty^{\frac{d}{d+\mu}}(x_\infty(s))\]
for some $c>0$ or equivalently, as long as $V_{\infty}(x_{\infty}(s))>0$,
\begin{equation}
\label{eq:finite-dec}
\frac{d}{ds}\left(V_{\infty}^\frac{\mu}{d+\mu}(x_{\infty}(s))\right) \leq -c\frac{\mu}{d+\mu}.
\end{equation}

%Since $\frac{d}{d+\mu} \in (0,1)$, we have finite-time convergence. 
Since $V_{\infty}(x_{\infty}(s))>0$ for $s\in [0,T_{x_0})$, by integrating  both sides of~\eqref{eq:finite-dec} on $[0,T_{x_0}]$ we get
\[0 = V_{\infty}^\frac{\mu}{d+\mu}(x_{\infty}(T_{x_0}))\leq V_{\infty}^\frac{\mu}{d+\mu}(x_0)-c\frac{\mu}{d+\mu}T_{x_0},\]
so that $T_{x_0}\leq \frac{d+\mu}{c\mu}\,V_{\infty}^\frac{\mu}{d+\mu}(x_0)$. 

Finally, the estimate~\eqref{eq:TCV} follows from Corollary~\ref{cor1} and the fact that $V_{\infty}^\frac{\mu}{d+\mu}$ is homogeneous of degree $\mu$ with respect to the family $(\delta_\epsilon^r)_{\epsilon>0}$.  
\end{proof}
\noindent
The following result asserts in particular that solutions of $(OCP)_\infty$ are also solutions of $(OCP)_T$ for $T$ large enough.

%Immediately, one deduces the following corollary [*** of Proposition~\ref{prop5} ***]
\begin{prop}
\label{cor3}
Let $x_\infty$ and $T_{x_0}$ as in Theorem~\ref{thm3} and denote as $(x_\infty,p_\infty)$ the associated extremal lift.
Then, $(x_\infty(t),p_\infty(t))=(0,0)$ for every $t\geq T_{x_0}$. Moreover, for every $T\geq T_{x_0}$ the restriction of $(x_\infty,p_\infty)$ to $[0, T]$ corresponds to the unique extremal lift associated with the solution of $(OCP)_{T}$.
%Let $(x^*,p^*)$ be the optimal pair solution of $(OCP)_\infty$ and $T_{x_0}$ be the first time when the optimal trajectory reaches $0$, then $(x^*,p^*)$ is the unique solution of $(OCP)_{T_{x_0}}$.
\end{prop}
\begin{proof}
The optimal trajectory $x_\infty$ and the associated optimal control are identically equal to zero for $t\geq T_{x_0}$. Then $\dot{p}_{\infty}(t) = -J_n^\top p_{\infty}(t)$ and $e_n^\top p_{\infty}(t) = 0$ for every $t\geq T_{x_0}$. Since the pair $(-J_n^\top,e_n^\top)$ is observable, we deduce that $p_{\infty}(t) = 0$ for every $t\geq T_{x_0}$. Then, the last part of the corollary follows from  Proposition~\ref{prop5}.
\end{proof}
\subsection{Implicit feedback expression}
%Within the same framework, 
We next study the regularity of the value function associated with 
$(OCP)_\infty$. %This will allow us to express the optimal control in feedback form.
The proof is postponed to a forthcoming publication.
\begin{prop}
    \label{prop14}
    The value function $V_\infty$ is continuously differentiable on $\mathbb{R}^n$ and $\nabla V_\infty(x_0)=-p_\infty(0)$, where $(x_\infty,p_\infty)$ is the extremal lift associated with the optimal trajectory $x_\infty$ starting at $x_0$.
\end{prop}
As a consequence of the previous result and thanks to~\eqref{eq1}, we have the following result.

\begin{cor}
\label{cor-feedback}
The optimal control $u_\infty$ for $(OCP)_\infty$  can be expressed in feedback form as 
\[u_\infty(x) = - \left\lfloor\frac{\partial V_\infty}{\partial x_n} (x)\right\rceil^{\frac{1}{q-1}}.\]
\end{cor}
Note that, as a consequence of Proposition~\ref{prop14}, the feedback expression of the optimal control determined in Corollary~\ref{cor-feedback} is continuous.
%%%%%%%%%%
%%%%%%%%%%

\subsection{Hamilton-Jacobi-Bellman equation}
\label{sec-hjb}
In this section, we provide the Hamilton-Jacobi-Bellman equation associated with $(OCP)_\infty$, and we prove that this partial differential equation admits a unique solution in the class of positive definite $C^1$ functions.

Since the value function $V_\infty$ associated with $(OCP)_\infty$ is of class $C^1$, we can apply \cite[Proposition 8.1.8]{sontag2013mathematical} as well as \cite[Proposition 3.2.1]{bertsekas2012dynamic} at once to obtain the following standard theorem.
\begin{thm}
\label{prop11}
The value function $V_\infty$ of $(OCP)_\infty$ satisfies, for every $x\in\mathbb{R}^n$, the Hamilton-Jacobi-Bellman equation
\begin{align}
\nabla V_\infty(x)^\top J_n x -\frac{q-1}{q}\left|\frac{\partial V_\infty}{\partial x_n}(x)\right|^\frac{q}{q-1} + F(x) = 0. \label{eq:HJB}
\end{align}
\end{thm}
%
%    \label{thm4}
%    Let us denote by [* NECESSARY? *]
%     \[\mathcal{A} \coloneqq \lbrace w \in L^{q}(\mathbb{R}_{\ge 0},\mathbb{R})|\quad \mathcal{J}_\infty(x_0,w) < \infty \rbrace\] for any $\tau \geq 0$, one has the following equality
%    \begin{align*}
%        &V_\infty(x_0)\\ 
%        &= \inf_{w \in \mathcal{A}} \left\{\int_0^\tau \mathcal{L}(\phi(w,x_0,s),w(s))ds + V_\infty(\phi(w,x_0,\tau))\right\}.
%    \end{align*}
%\end{thm}
%\begin{sketch}
%        Introducing \[\tilde{V} = \inf_{w \in \mathcal{A}} \int_0^\tau \mathcal{L}(\phi(w,x_0,s),w(s))ds + V_\infty(\phi(w,x_0,\tau)),\]
%        one obtains immediately $\forall x \in \mathbb{R}^n, \quad V_\infty(x) \geq \tilde{V}(x)$ since the optimal control is admissible. The converse inequality is proved by exploiting the definition of infimum.
%    \end{sketch}
%\noindent
%We will use the above theorem to prove the next proposition.
%\begin{prop}
%\label{prop11}
%The Hamilton-Jacobi-Bellman equation is written in the following way :
%\begin{align}
%    \langle \nabla V_\infty(x_0),J_nx_0\rangle -\frac{q-1}{q}\left|\frac{\partial V_\infty}{\partial x_n}(x_0)\right|^\frac{q}{q-1} + F(x_0) = 0. \label{eq:HJB}
%\end{align}
%\end{prop}
%\begin{sketch}
%   The equation follows after the direct differentiation of the equality from Theorem~\ref{thm4} and passing to the limit in $\tau$.
%\end{sketch}
%Furthermore, a solution of the Hamilton-Jacobi-Bellman equation must be unique.
The next result will be used to deduce the uniqueness of the solution of~\eqref{eq:HJB}. The proof is postponed to a forthcoming paper.
\begin{prop}
    \label{prop13}
    Let $W$ be a positive definite function in $\mathcal{C}^1(\mathbb{R}^n,\mathbb{R})$,
    \begin{enumerate}
    \setlength{\itemsep}{5pt}
        \item if $W$ is a sub-solution of (\ref{eq:HJB}), i.e., if it satisfies for every $x \in \mathbb{R}^n$ the inequality
        \[\nabla W^\top(x)J_n x -\frac{q-1}{q}\left|\frac{\partial W}{\partial x_n}(x)\right|^\frac{q}{q-1} \leq -F(x)  \]
        \noindent
        then  $V_\infty(x) \leq W(x)$  for every $x \in \mathbb{R}^n$,
        \item if $W$ is a super-solution of (\ref{eq:HJB}), i.e., if it satisfies for every $x \in \mathbb{R}^n$ the inequality
        \[\nabla W^\top (x)J_n x -\frac{q-1}{q}\left|\frac{\partial W}{\partial x_n}(x)\right|^\frac{q}{q-1} \geq -F(x)  \]
        then $V_\infty(x) \geq W(x)$ for every $x \in \mathbb{R}^n$.
    \end{enumerate}
\end{prop}

From Proposition~\ref{prop13}, we deduce the following corollary.
\begin{cor}
    \label{cor6}
    The following statements hold true.
    \begin{enumerate}
    \setlength{\itemsep}{5pt}
        \item The function $V_\infty$ is the unique solution of (\ref{eq:HJB}) among positive definite $C^1(\mathbb{R}^n,\mathbb{R})$ functions.
        \item Let $W$ be a positive definite mapping in $\mathcal{C}^1(\mathbb{R}^n,\mathbb{R})$
    such that the function
     \[x\mapsto \frac{q-1}{q}\left|\frac{\partial W}{\partial x_n}(x)\right|^\frac{q}{q-1} - \nabla W(x)^\top J_nx\]
    satisfies \text{(A1)}--\text{(A3)}. Then the feedback control \[u_W(x)=-\left\lfloor\frac{\partial W}{\partial x_n}(x)\right\rceil^\frac{1}{q-1}\] induces finite time convergence.
       % Let $W$ be a positive definite $\mathcal{C}^1(\mathbb{R}^n,\mathbb{R})$ mapping and $F$ be a mapping satisfying the assumptions \text{(A1)}--\text{(A3)} and such that $W,F$ satisfy (\ref{eq:HJB}). Then $W$ is the value function corresponding to the infinite horizon cost 
    %\[\int_0^\infty \frac{|u(s)|^{q}}{q}+F(x(s)) ds,\]
    %\noindent
    %and the feedback control 
    %\[u(x)=-\left\lfloor\frac{\partial W}{\partial x_n}(x)\right\rceil^\frac{1}{q-1}\] 
    %nduces finite time convergence. 
    \end{enumerate}
\end{cor}
%%%%%%%%%%
%%%%%%%%%%

\subsection{Existing feedback and optimal control}

%%%%%%%%%%
%%%%%%%%%%
%\subsection{Application to gain regulation {\color{blue}[* change title???? *]}}
As an application, we will establish a link between the results in \cite{hong2002finite} and those of the previous sections.
First, we recall the expression of the controllers in \cite{hong2002finite}. Let $r_1,k > 0$ be such that $r_1 > (2n-1)k$, and set
\[r_i =r_1-(i-1)k,\quad i=2,\dots,n.\]
Furthermore, let $ \beta_0=r_2+1,\ \beta_{n-1}=2$, and
\[\beta_j=\beta_0\frac{r_1}{r_{j+1}},\quad j=1,\dots,n-2.\]
%\begin{align}
%    &r_i =r_1-(i-1)k,\quad i=1,\dots,n\label{eq37} \\
%    & \beta_0=r_2,\quad \beta_{n-1}=1\\
%    &\forall j \in \lbrace 0,\dots, n-2\rbrace,\quad \beta_j = \frac{(1+\beta_{j-1})r_{j}-r_{j+1}}{r_{j+1}}.\label{eq38}
%\end{align} 
Introduce $V_1: y \in \mathbb{R} \mapsto |y|^{\beta_0}$ as well as $(\ell_j)_{1 \leq j \leq n}$ a family of large enough positive constants. Let  $u_0=0$. For $1 \leq j \leq n$ and $2\leq k\leq n$, we set $\bar{x}_j=(x_1,\dots,x_j)$ and we define recursively the following functions
\begin{align}
    u_j(\bar{x}_j) &= -\ell_j\left\lfloor  \lfloor x_j \rceil^{\beta_{j-1}-1} -\lfloor u_{j-1}(\bar{x}_{j-1})\rceil^{\beta_{j-1}-1} \right\rceil^\frac{r_j-k}{r_j(\beta_{j-1}-1)}, \label{eq39}\\
    W_j(\bar{x}_j)&=\frac{1}{\beta_{j-1}}\left(|x_j|^{\beta_{j-1}}+(\beta_{j-1}-1)|u_{j-1}(\bar{x}_{j-1})|^{\beta_{j-1}}\right) \nonumber \\
    &-x_j\lfloor u_{j-1}(\bar{x}_{j-1}) \rceil^{\beta_{j-1}-1}, \label{eq40}\\
     V_k(\bar{x}_k) &= W_k(\bar{x}_k) + V_{k-1}(\bar{x}_{k-1}). \label{eq41}
\end{align}
\noindent
In \cite{hong2002finite} it was shown that the feedback controller $u_n$ stabilizes the integrator chain in finite time. Furthermore, from (\ref{eq40}) and (\ref{eq41}), it is clear that
\[u_n(x) = -\ell_n\left\lfloor  \frac{\partial V_n}{\partial x_n}(x)\right\rceil^\frac{1}{q-1}\]
with $q = \frac{2r_n-k}{r_n-k}$.
Moreover, the following inequality holds
\begin{align}
    \nabla V_n(x)^\top J_n x - \ell_n\left|\frac{\partial V_n(x)}{\partial x_n }\right|^{\frac{q}{q-1}} < 0. \label{eq72}
\end{align}
Introduce the mapping $W = \left(\frac{q}{q-1}\ell_n\right)^{q-1}V_n$.
We have 
\begin{align*}
   &  \nabla W(x)^\top J_n x - \frac{q-1}{q}\left|\frac{\partial W(x)}{\partial x_n}\right|^{\frac{q}{q-1}}\\
    &= \left(\frac{q}{q-1}\ell_n\right)^{q-1} \Big(\nabla V_n(x)^\top J_n x - \ell_n\left|\frac{\partial V_n(x)}{\partial x_n} \right|^{\frac{q}{q-1}}\Big).
\end{align*}
By \eqref{eq72}
%\begin{align*}
%    \langle \nabla W(x), J_n x\rangle &= \left(\frac{q}{q-1}\ell_n\right)^{q-1}\langle \nabla V_n(x), J_n x\rangle \\
%    &< \left(\frac{q}{q-1}\ell_n\right)^{q-1}\ell_n\left|\frac{\partial V_n(x)}{\partial x_n} \right|^{\frac{q}{q-1}}\\
%    &= \frac{q-1}{q}\left|\frac{\partial W(x)}{\partial x_n}\right|^{\frac{q}{q-1}}.
%\end{align*}
 the mapping $x\mapsto \frac{q-1}{q}\left|\frac{\partial W(x)}{\partial x_n }\right|^{\frac{q}{q-1}}- \nabla W(x)^\top J_n x $ is positive definite. Furthermore, it is continuous and homogeneous of degree $d = 2r_n-k$ with respect to the family of dilation $(\delta^r_\epsilon)_{\epsilon >0}$, as shown in~\cite{hong2002finite}. By Corollary~\ref{cor1} there exist $c_1,c_2  > 0$ such that for any $x \in \mathbb{R}^n$
\begin{equation}
    F_1(x)\leq \frac{q-1}{q}\left|\frac{\partial W(x)}{\partial x_n }\right|^{\frac{q}{q-1}} \!\! - \nabla W(x)^\top J_n x \leq F_2(x) \label{eq73}
\end{equation}
\noindent
with $\alpha_i=\frac{d}{r_i}$ and $F_k(x) = c_k\sum_{i=1}^n |x_i|^{\alpha_i}$ with $k=1,2$.

For $i \in \lbrace1,2\rbrace$, denote by $V_{i,q}$ the unique solution of (\ref{eq:HJB}) with $F=F_i$. From (\ref{eq73}), it is clear that $W$ is a sub-solution of (\ref{eq:HJB}) with $F=F_1$ and also a super-solution of (\ref{eq:HJB}) with $F=F_2$ so that using Proposition~\ref{prop13}, one has
\begin{align}
    \forall x \in \mathbb{R}^n, \quad V_{1,q}(x) \leq W(x) \leq V_{2,q}(x). \label{eq74}
\end{align}
\noindent
%In particular, if $x_0 \neq 0$, one has from (\ref{eq74})
%\[\frac{q-1}{q}\left(\frac{V_{1,q}(x_0)}{V_n(x_0)}\right)^\frac{1}{q-1} \leq \ell_n \leq \frac{q-1}{q}\left(\frac{V_{2,q}(x_0)}{V_n(x_0)}\right)^\frac{1}{q-1}.\]
\section{Numerical simulations}
\label{sec-simu}
To illustrate finite-time stabilization, we consider~\eqref{eq:L} for $n=6$, choose $(d, \mu) = (3, 0.2)$ and $q=2$, with $$F: x \in \mathbb{R}^6 \mapsto |x_1|^\frac{10}{9}+|x_2|^\frac{6}{5} + |x_3|^\frac{30}{23} + |x_4|^\frac{10}{7} + |x_5|^\frac{30}{19} + |x_6|^\frac{30}{17}.$$ According to Proposition~\ref{prop3}, $F$ is weighted-homogeneous of degree $d$ with weights $r=(2.7,2.5,2.3,2.1,1.9,1.7)$. Notice that the choice of the parameters $d,\mu$ and $q$ is totally arbitrary as long as conditions from Proposition~\ref{prop3} hold and $q > 1$. The simulations are performed in \textit {Julia} using the package OptimalControl.jl, cf.~\cite{Caillau_OptimalControl_jl_a_Julia}. %,  based on the IPOPT solver\footnote{\texttt{https://github.com/JuliaSmoothOptimizers/NLPModelsIpopt.jl}}. %(see e.g. \cite[OptimalControl.jl, NLPModelsIpopt.jl]{Caillau_OptimalControl_jl_a_Julia}). 
For a sufficiently large time horizon and a given initial condition, the problem $(OCP)_T$ is transcribed into a nonlinear programming problem and solved numerically with the built-in IPOPT solver. The discretization step is $0.08$.
\begin{figure}[h]
    \centering
    \includegraphics[width=0.95\columnwidth]{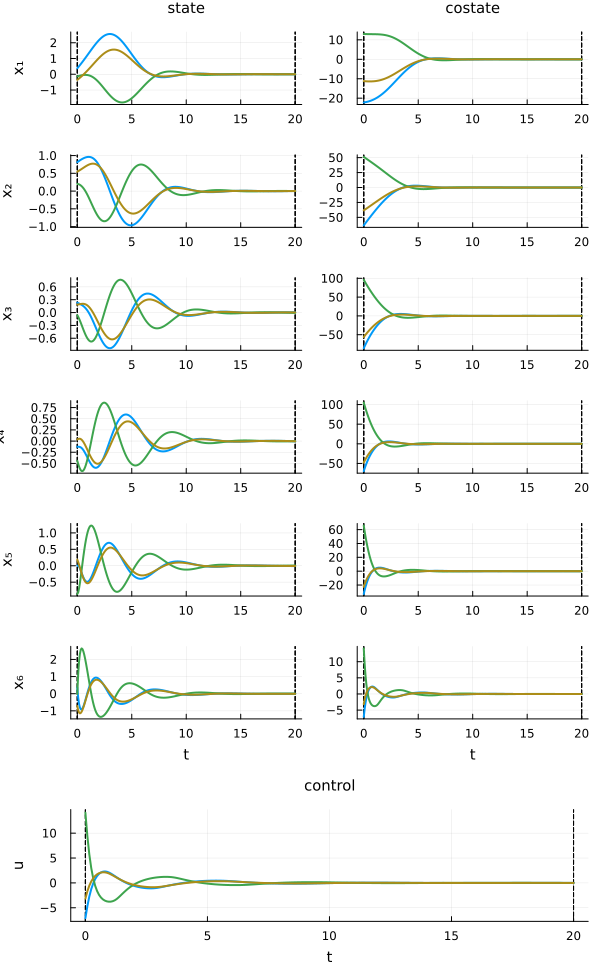}
    \caption{Optimal state, costate trajectory and control for three randomly sampled initial conditions from the unit Euclidean sphere.}
    \label{fig:Simu2}
\end{figure}

Figure~\ref{fig:Simu2} shows the behavior of the optimal state trajectory, costate trajectory, and control for some initial conditions randomly sampled in the unit sphere. The color of each curve corresponds to a specific initial condition. Note that computing several thousand solutions corresponding to initial conditions uniformly sampled on the unit sphere requires only about two minutes for an integrator chain of length six. The first execution, which may take a few seconds, is the slowest due to package loading %and problem compilation, 
whereas subsequent executions run in the millisecond range.

%\begin{figure}[h]
%    \centering
%    \includegraphics[width=\linewidth]{[-1,1].png}
%    \caption{Minimizing control and optimal pair for $(OCP)_\infty$ starting from $x_0=(-1,1)^\top$.}
%    \label{fig:Simu2}
%\end{figure}
%\begin{figure}[!t]
%    \centering
%    \includegraphics[width=\linewidth]{[100,100].png}
%    \caption{Minimizing control and optimal pair for $(OCP)_\infty$ starting from $x_0=(100,100)^\top$.}
%    \label{fig:Simu3}
%\end{figure}
%\begin{figure}
%    \centering
%    \includegraphics[width=\linewidth]{DIphaseportrait.png}
 %   \caption{Phase portrait of the state trajectory for two hundred initial conditions}
 %   \label{fig3}
%\end{figure}
\FloatBarrier
\section{Conclusion}
\label{sec-conclusion}
This paper presents a new class of finite-time stabilizing controllers for the pure integrator chain obtained as optimal controls associated with a large class of optimal control problems. These controllers admit a feedback form in terms of the value function of the corresponding 
optimal control problem. Furthermore, the value function is shown to be regular enough to satisfy, in a strong sense, the corresponding Hamilton-Jacobi-Bellman equations. Future work will address extensions to fixed-time stabilization, design of 
homogeneity-based algorithms to approximate the solution of the resulting Hamilton-Jacobi-Bellman equation, and robustness analysis under model uncertainties.
%\bibliography{ifacconf}              
%\nocite{*}

\end{document}